\documentclass[12pt]{amsart}

\usepackage{enumitem}
\usepackage{stmaryrd}
\usepackage{url}
\usepackage[dvipsnames]{xcolor}
\usepackage{tikz-cd}
\usepackage{graphicx}
\usepackage{amsmath}
\usepackage{color}
\usepackage{caption}
\usepackage{subcaption}
\usepackage{amsfonts,amssymb,amscd}
\usepackage{mathrsfs}
\usepackage{bm}
\usepackage{verbatim} 
\usepackage{manfnt}
\usepackage{yhmath}

\usepackage[toc,page]{appendix}

\calclayout

\newtheoremstyle{remboldstyle}
  {}{}{\itshape}{}{\bfseries}{.}{.5em}{{\thmname{#1 }}{\thmnumber{#2}}{\thmnote{ (#3)}}}
\theoremstyle{remboldstyle}

\usepackage[outdir=./]{epstopdf}

\newtheorem{thm}{Theorem}[section]
\newtheorem{prop}[thm]{Proposition}
\newtheorem{lem}[thm]{Lemma}

\newtheorem{question}{Question}

\newtheorem{thmx}{Theorem}

\theoremstyle{definition}

\newtheorem{rem}[thm]{Remark}
\newtheorem{notation}[thm]{Notation}

\usepackage[greek,english]{babel}
\usepackage[LGR,T1]{fontenc}
\usepackage[utf8]{inputenc}

\makeatletter
\DeclareFontFamily{U}{tipa}{}
\DeclareFontShape{U}{tipa}{m}{n}{<->tipa10}{}
\newcommand{\arc@char}{{\usefont{U}{tipa}{m}{n}\symbol{62}}}%

\newcommand{\arc}[1]{\mathpalette\arc@arc{#1}}

\newcommand{\arc@arc}[2]{%
  \sbox0{$\m@th#1#2$}%
  \vbox{
    \hbox{\resizebox{\wd0}{\height}{\arc@char}}
    \nointerlineskip
    \box0
  }%
}
\makeatother

\numberwithin{equation}{section}

\catcode`\@=11
\newdimen\cdsep
\def\cdstrut{\vrule height .6\cdsep width 0pt depth .4\cdsep}
\def\@cdstrut{{\advance\cdsep by 2em\cdstrut}}

\def\arrow#1#2{
  \ifx d#1
    \llap{$\scriptstyle#2$}\left\downarrow\cdstrut\right.\@cdstrut\fi
  \ifx u#1
    \llap{$\scriptstyle#2$}\left\uparrow\cdstrut\right.\@cdstrut\fi
  \ifx r#1
    \mathop{\hbox to \cdsep{\rightarrowfill}}\limits^{#2}\fi
  \ifx l#1
    \mathop{\hbox to \cdsep{\leftarrowfill}}\limits^{#2}\fi
}
\catcode`\@=12

\newcommand{\Chat}{\widehat{\mathbb{C}}}

\DeclareMathOperator{\dist}{dist}

\makeatletter
\@namedef{subjclassname@2020}{\textup{2020} Mathematics Subject Classification}
\makeatother

\begin{document}


\title[Simultaneous Approximation by Attracting Basins]{Simultaneous Approximation by Attracting Basins}

\author{Colton Fisher, Aiden Hill, Kirill Lazebnik, Palmer Thompson}
\thanks{\noindent The authors were supported by NSF grant DMS-2452130.} 





\begin{abstract} We show that any $d\geq3$ pairwise-disjoint open sets $A_1$, ..., $A_d\subset\Chat$ sharing a common boundary $J$ can be simultaneously approximated by the $d$ attracting basins $\mathcal{A}_1$, ..., $\mathcal{A}_d$ of a rational map $r$ having Fatou set $\mathcal{F}(r)=\mathcal{A}_1\sqcup...\sqcup\mathcal{A}_d$ and so that the Julia set $\mathcal{J}(r)$ approximates $J$. 
\end{abstract}


\maketitle


\section{Introduction}

The field of Complex Dynamics centers around the study of dynamical systems of the form
\begin{equation}\label{rational} r: \Chat \rightarrow \Chat \end{equation}
where $r=r(z)$ is a rational map in one complex variable (the quotient of two polynomials $r(z)=p(z)/q(z)$ having no roots in common), and $\Chat:=\mathbb{C}\cup\{\infty\}$ is the Riemann sphere, obtained by adjoining $\infty$ to the complex plane $\mathbb{C}$. The two early 20th century pioneers of the field, Pierre Fatou and Gaston Julia, both arrived at the following dynamical decomposition of $\Chat$ associated to (\ref{rational}): the \emph{Fatou set}, denoted $\mathcal{F}(r)$, consists of points in a neighborhood of which the iterates $r^n:=r\circ ... \circ r$ form a normal family in the sense of Montel (see \cite{MR4422101}, Chapter 5), and the complement of $\mathcal{F}(r)$ is called the \emph{Julia set}, denoted $\mathcal{J}(r)$. This definition lies at the center of classical and modern work in the field.



A point $\zeta\in\Chat$ for which $r(\zeta)=\zeta$ and $|r'(\zeta)|<1$ is called an \emph{attracting fixed point} of $r$. If $\zeta$ is an attracting fixed point of $r^n$ for $n\geq1$, we call $\zeta$, $r(\zeta)$, ...., $r^n(\zeta)=\zeta$ an \emph{attracting cycle}, and $\zeta$ an \emph{attracting periodic point}. The \emph{basin of attraction} $\mathcal{A}(\zeta)$ for an attracting periodic point $\zeta$ consists of points which converge, under iteration, to the attracting cycle. Systems $r$ for which $\mathcal{F}(r)$ consists only of basins of attraction are called \emph{hyperbolic}; a central conjecture in the field of complex dynamics is that hyperbolic rational maps are dense among all rational maps. 



If $r$ is hyperbolic, so that $\mathcal{F}(r)=\mathcal{A}(\zeta_1)\sqcup...\sqcup\mathcal{A}(\zeta_d)$ for attracting periodic points $\zeta_1$, ..., $\zeta_d$, then each $\mathcal{A}(\zeta_j)$ must share a common boundary = $\mathcal{J}(r)$ (see  \cite{MR1230383}, Chapter 3). The question studied in this paper is a sort of approximate converse: can any $d$ pairwise-disjoint open sets sharing a common boundary be approximated by the attracting basins of a hyperbolic map? More precisely, we study the following question posed to us by Malik Younsi:

\begin{question}\label{our_quest} Suppose $A_1$, ..., $A_d\subset\Chat$ are open sets all sharing a common boundary $J$. Does there exist an $r$ with attracting periodic points $\zeta_1$, ..., $\zeta_d$ and $\mathcal{F}(r)=\mathcal{A}(\zeta_1)\sqcup...\sqcup\mathcal{A}(\zeta_d)$ so that one has the simultaneous approximation $J \approx \mathcal{J}(r)$ and $A_j\approx \mathcal{A}(\zeta_j)$ for each $j$?
\end{question}

\noindent In this paper, the notion of approximation $\approx$ we will use is the Hausdorff metric; we recall the definition. 

\begin{notation} All distances between two points in $\Chat$ will be measured with respect to the spherical metric, denoted by $d(\cdot,\cdot)$. We denote the Hausdorff distance between two sets $A$, $B\subset\Chat$ by $d(A,B)$, and recall that $d(A,B)<\varepsilon$ if and only if $A\subset N_\varepsilon(B):=\{z \in \Chat : d(z,B)<\varepsilon\}$ and $B\subset N_\varepsilon(A)$. 
\end{notation}

\noindent Our main result is the following answer to Question \ref{our_quest}. 

\begin{thmx}\label{our_approx_result} Let $\varepsilon>0$, $d\geq2$ and $A_1$, ..., $A_d$ be disjoint open sets in $\Chat$ sharing a common boundary $J$. Then there exists a rational map $r$ with attracting fixed points $\zeta_1$, ..., $\zeta_d$ so that $\mathcal{F}(r)=\mathcal{A}(\zeta_1)\sqcup...\sqcup\mathcal{A}(\zeta_d)$ and $d_H(J,\mathcal{J})$, $d_H(A_1,\mathcal{A}_1)$, ...., $d_H(A_d,\mathcal{A}_d)$ are all $<\varepsilon$. 
\end{thmx}

Theorem \ref{our_approx_result} was influenced by some closely related work in the literature which we now briefly survey. Under the assumption that $d=2$ and $J$ is a Jordan curve, Theorem \ref{our_approx_result} is a result of \cite{MR3377290}. When $d=2$ and $J$ is a pairwise-disjoint, finite collection of mutually-exterior Jordan curves, Theorem \ref{our_approx_result} is a result of \cite{MR3955554}. In both \cite{MR3377290}, \cite{MR3955554}, $r$ may be taken to be a polynomial (provided $J$ does not contain $\infty$). When $d=2$ and $J$ is a pairwise-disjoint, finite collection of Jordan curves (not necessarily mutually exterior), Theorem \ref{our_approx_result} is a result of \cite{MR4880202}, in which case $r$ must be non-polynomial in general (by a result of \cite{MR3955554}). 

Other recent relevant work includes a study of the rate of approximation in Theorem \ref{our_approx_result} (for $d=2$ and $J$ a pairwise-disjoint, finite collection of mutually-exterior Jordan curves) in \cite{MR3794102}. Approximation by dendrite Julia sets was studied in \cite{MR3420484}. Related work in the transcendental setting includes \cite{MR4375923} and \cite{MR4846778}. 

Our main contribution is in being able to simultaneously approximate $d>2$ open sets as in Theorem \ref{our_approx_result}; this requires some different ideas from the $d=2$ case. To highlight the difference between $d=2$ and $d>2$, let us remark that while it is not difficult for $d=2$ disjoint open sets to share a common boundary (for instance the two complementary components of any Jordan curve), the common boundary of $d>2$ open sets must be comparatively intricate; examples include the Lakes of Wada and attracting basins of many Newton maps. 

A sketch of our proof of Theorem \ref{our_approx_result} is as follows. We replace the given $A_1$, ..., $A_d$ by open subsets $A_1^\varepsilon\subset A_1$, ..., $A_d^\varepsilon\subset A_d$ which still approximate $A_1$, ..., $A_d$ but now $\overline{A_1^\varepsilon}$, ..., $\overline{A_d^\varepsilon}$ can be properly separated by pairwise-disjoint open sets. This means that Runge's Theorem applied to a piecewise-constant function $z\mapsto w_j$ in $A_j^\varepsilon$ produces a rational approximation $R$ to this piecewise-constant function, and the constants $w_j$ may also be chosen to lie inside the open sets $A_j^\varepsilon$. In this way we can guarantee attracting fixed points $w_1$, ..., $w_d$ containing each of $A_1^\varepsilon$, ..., $A_d^\varepsilon$ in their basin of attraction, but we need to work harder to guarantee that there are no other Fatou components. To this end, we consider $r:=P_n\circ R$, where $P_n$ is the $n$th iterate of a polynomial whose only critical points are $w_1$, ..., $w_d$, and so that each of $w_1$, ..., $w_d$ are fixed by $P$. After ensuring that no critical values of $R$ lie in $\mathcal{J}(P_n)$, one is able to deduce that for all large enough $n$, all critical values of $P_n\circ R$ lie in basins of attraction of (perturbations of) the attracting fixed points $w_1$, ... $w_d$; this ensures that $P_n\circ R$ has no other Fatou components. In Section \ref{main_section} we will provide the details.




\vspace{5mm}

\noindent \emph{Acknowledgements.} We thank Malik Younsi for posing Question \ref{our_quest} to us.

\section{Proof of Theorem \ref{our_approx_result}.}\label{main_section}

\begin{notation} Let $A_1$, ..., $A_d$ be $d\geq2$ pairwise-disjoint open sets sharing a common boundary $J$, so that $\Chat=A_1\sqcup...\sqcup A_d\sqcup J$.  We will assume without loss of generality that $\infty\in J$.
\end{notation}

\begin{lem}\label{covering_lemma} Let $A\subset\Chat$ be open and $\varepsilon>0$. Then there exists an open subset $B^\varepsilon\subset A$ with $\partial B^{\varepsilon}\cap \partial A=\emptyset$ so that for every $z\in\partial A$, one has $D(z,\varepsilon)\cap B^\varepsilon\not=\emptyset$.
\end{lem}

\begin{proof} For each $z\in\partial A$, choose $w(z) \in A\cap D(z,\varepsilon)$ and $r_{w(z)}>0$ so that $D(w(z),2r_{w(z)})\subset A\cap D(z,\varepsilon)$. Then 
\[ \partial A\subset \cup_{z\in\partial A} D(w(z),\varepsilon) \] 
so that by compactness of $\partial A$, there exists a finite subcover 
\[ \partial A\subset \cup_{j=1}^n D(w(z_j), \varepsilon). \] 
The set $B^\varepsilon:=\cup_{j=1}^nD(w(z_j), r_{w(z_j)})$ has the desired properties.
\end{proof}

\begin{notation}
We let 
\begin{equation} P(z):=z^{d}+\frac{d}{d-1}z, \end{equation}
and $P_\lambda(z):=\lambda P(z/\lambda)$ for $\lambda>0$.
\end{notation}

\begin{rem} It is readily checked that $P$ has a fixed critical point at $\infty$, and $d-1$ finite critical points each of which are fixed by $P$; we will denote these $d$ fixed critical points (including $\infty$) by $\xi_1$, $\xi_2$, ..., $\xi_{d}$. Thus $P_\lambda$ has $d$ fixed critical points $\lambda\xi_1$, $\lambda\xi_2$, ..., $\lambda\xi_d$. Let $\varepsilon>0$ and $B_j^\varepsilon\subset A_j$ be the open set obtained by applying Lemma \ref{covering_lemma} to $A_j$. Recall $\infty\in J$. We fix $\lambda\gg 0$ so that $d(\lambda\xi_j, \infty)<\varepsilon/2$ for each $j$, and so that the points $\lambda\xi_j$ are not contained in $\cup_jB_j^\varepsilon$. We fix $\delta>0$ so that the discs $D(\lambda\xi_j, \delta)$ are pairwise-disjoint and also disjoint from $B_j^\varepsilon$, and moreover 
\begin{equation}\label{delta_contracting} P_\lambda(D(\lambda\xi_j, \delta)) \subset D(\lambda\xi_j, \delta/2),
\end{equation}
for each $j$, where we remark that (\ref{delta_contracting}) holds for sufficiently small $\delta$ since each $\lambda\xi_j$ is an attracting fixed point of $P_\lambda$. 
\end{rem}

\begin{notation} For $\varepsilon>0$, we denote by 
\begin{equation} A_j^\varepsilon := \{z \in A_j : \dist(z, J)>\varepsilon\}\cup B_j^\varepsilon \cup D(\lambda\xi_j, \delta).
\end{equation} 
\end{notation}

\begin{lem}\label{good_haus_approx} We have that $d_H(A_j^\varepsilon, A_j)\xrightarrow{\varepsilon\rightarrow0}0$.
\end{lem}
\begin{proof} Let us first show that each point $z\in A_j^\varepsilon$ is close to $A_j$. Well if $z\in B_j^\varepsilon$ then $z\in A_j$, and if $z\in D(\lambda\xi_j, \delta)$ then $z$ is close to $\infty\in\partial A_j$ so $z$ is close to $A_j$. 

Now let us show that each point $z\in A_j$ is close to $A_j^\varepsilon$. If $\dist(z, J)>\varepsilon$ then evidently $z\in A_j^\varepsilon$, and if $\dist(z, J)<\varepsilon$ then $z$ is close to $B_j^\varepsilon\subset A_j^\varepsilon$ by Lemma \ref{covering_lemma}. 
\end{proof}

\begin{rem} For a rational map $R:\Chat\rightarrow\Chat$, we denote the critical points of $R$ by $\textrm{CP}(R)$, and the critical values of $R$ by $\textrm{CV}(R)$.
\end{rem}

\begin{prop}\label{approximation_step} For all $\varepsilon>0$, there exists a rational map $R$ so that 
\begin{equation}\label{classical_runge} R(A_j^\varepsilon)\subset D(\lambda\xi_j, \delta).
\end{equation}
Moreover, $R$ may be chosen so that 
\begin{equation}\label{nocvinjulia} \emph{CV}(R)\cap\mathcal{J}(P_\lambda)=\emptyset.
\end{equation}
\end{prop}

\begin{proof} The piecewise-defined (piecewise-constant) function $z\mapsto \lambda\xi_j$ in $A_j^\varepsilon$ is holomorphic in a neighborhood of $\cup_j \overline{A_j^\varepsilon}$, so Runge's Theorem applies to produce a rational $S$ satisfying 
\begin{equation}\label{class2} S(A_j^\varepsilon)\subset D(\lambda\xi_j, \delta/2).\end{equation}
If $\textrm{CV}(S)\cap\mathcal{J}(P_\lambda)=\emptyset$, we set $R:=S$ and we are finished. 


If $\textrm{CV}(S)\cap\mathcal{J}(P_\lambda)\not=\emptyset$, then consider $R_\eta:=(1+\eta)S$ where $\eta$ is a $\mathbb{C}$-parameter with $|\eta|$ small. By (\ref{class2}), We have that (\ref{classical_runge}) still holds with $R:=R_\eta$ for all sufficiently small $|\eta|$, and moreover we claim that there exists $\eta$ so that 
\begin{equation}\label{SBWOC} \textrm{CV}(R_\eta)\cap\mathcal{J}(P_\lambda)=\emptyset. \end{equation}

Indeed, note that for each critical value $\zeta$ of $R_0=S$, the map $R_\eta$ has critical value $(1+\eta)\zeta$, and $(1+\eta)\zeta$ traces out a small ball centered at $\zeta$ for variable $\eta$. Thus, if we suppose by way of contradiction that (\ref{SBWOC}) fails for all small $\eta$, this would mean that $\mathcal{J}(P_\lambda)$ has positive area, and this is impossible since $P_\lambda$ is hyperbolic as all critical points are fixed (Julia sets of hyperbolic polynomials have zero area: see for instance Theorem V.2.3 of \cite{MR1230383}).
\end{proof}

\begin{notation} We will consider the mapping
\begin{equation} S_n:=P_\lambda^n\circ R,
\end{equation}
where $P_\lambda^n$ denotes the $n^{\textrm{th}}$ iterate of $P_\lambda$.
\end{notation}

\begin{lem}\label{afpfors} For all $n$, the map $S_n$ has an attracting fixed point $\zeta_j \in D(\lambda\xi_j, \delta)$ for each $j$, and moreover the disc $D(\lambda\xi_j, \delta/2)$ is contained inside the basin of attraction for $\zeta_j$.
\end{lem}

\begin{proof} By (\ref{classical_runge}), we have that 
\[ R(D(\lambda\xi_j, \delta))\subset D(\lambda\xi_j, \delta),\] 
and so by (\ref{delta_contracting}) we conclude that 
\[ S_n(D(\lambda\xi_j, \delta)) \subset D(\lambda\xi_j, \delta/2).\]
Thus $S_n$ has an attracting fixed point $\zeta_j \in D(\lambda\xi_j, \delta/2)$. 

Moreover, since 
\[ S_n(D(\lambda\xi_j, \delta/2)) \subset D(\lambda\xi_j, \delta/2),\]
the map $S_n$ is a strict contraction with respect to the hyperbolic metric on $D(\lambda\xi_j, \delta/2)$, and so the orbits of all points in $D(\lambda\xi_j, \delta/2)$ must converge to a common limit, and since $\zeta_j\in D(\lambda\xi_j, \delta/2)$ is an attracting fixed point, this common limit must be $\zeta_j$.
\end{proof}

\begin{notation} We will continue to denote the attracting fixed points of $S_n$ in the conclusion of Lemma \ref{afpfors} by $\zeta_1$, ..., $\zeta_d$. 
\end{notation}

\begin{rem}\label{standard_but_deep} The proof of Proposition \ref{no_other_basins_prop} below relies on the following standard but deep fact: if all critical values of a rational map $r$ lie in a basin of attraction of an attracting fixed point, then the Fatou set $\mathcal{F}(r)$ consists only of these basins of attraction. Indeed, the components of $\mathcal{F}(r)$ must be pre-periodic by Sullivan's no wandering domain Theorem, and pre-periodic components admit a classification (see Chapter 15 of \cite{MilnorCDBook}); moreover every pre-periodic component in this classification must have an associated critical value orbit (for detailed statements, see Theorems 8.6 and 11.17, Corollary 10.11 and Lemma 15.7 of [Mil06]). Thus, if all critical values of $r$ lie in a basin of attraction, there can be no other Fatou components apart from these basins of attraction. 
\end{rem}

\begin{prop}\label{no_other_basins_prop} For all sufficiently large $n$, the Fatou set of the map $S_n$ consists only of the attracting basins for $\zeta_1$, ..., $\zeta_d$.
\end{prop}

\begin{proof} Since each critical point of $P_\lambda$ is fixed, $\mathcal{F}(P_\lambda)$ consists only of the attracting basins for $\lambda\xi_1$, ..., $\lambda\xi_d$. By (\ref{nocvinjulia}),  it follows that each critical value of $R$ is contained in one of these attracting basins for $\lambda\xi_1$, ..., $\lambda\xi_d$. Thus, by taking $n$ large enough, we have that $P_\lambda^n(\textrm{CV}(R))$ is contained in a union of $d$ discs of arbitrarily small radius centered at $\lambda\xi_1$, ..., $\lambda\xi_d$. 

The critical values of $S_n$ consist of $P_\lambda^n(\textrm{CV}(R))$ together with the critical values $\lambda\xi_1$, ..., $\lambda\xi_d$ of $P_\lambda$; thus for all sufficiently large $n$, we have that the critical values of $S_n$ are contained in discs of arbitrarily small radius centered at one of the $\lambda\xi_j$. Hence, by Lemma \ref{afpfors}, for all large $n$ we have that $\textrm{CV}(S_n)$ lies in the union of the attracting basins for $\zeta_j$ under $S_n$. Hence, recalling Remark \ref{standard_but_deep}, we see that there can be no other Fatou components for $S_n$ other than the attracting basins for the $\zeta_j$.
\end{proof}

\begin{notation} We denote the attracting basins for the attracting fixed points $\zeta_1$, ..., $\zeta_d$ of $S_n$ by $\mathcal{A}_1$, ..., $\mathcal{A}_{d}$.
\end{notation}

\noindent Theorem \ref{our_approx_result} follows directly from the following.

\begin{thm} Let $\varepsilon>0$. For all sufficiently large $n$, the Fatou set of the map $S_n$ consists of the $d$ attracting basins $\mathcal{A}_1$, ..., $\mathcal{A}_{d}$ with common boundary $\mathcal{J}:=\mathcal{J}(S_n)$, and 
\begin{equation}\label{haus_ineq} d_H(A_j,\mathcal{A}_j)<\varepsilon \text{\;\;\;for\;\;\;}j=1,2,...,d \text{\;\;\;and\;\;\;}d_H(J,\mathcal{J})<\varepsilon.
\end{equation}
\end{thm}

\begin{proof} The first assertion was proven in Proposition \ref{no_other_basins_prop}; it remains to prove (\ref{haus_ineq}).

Let $z\in J$. Then $D(z,\varepsilon)$ must intersect two different $A_j^\varepsilon$, and each $A_j^\varepsilon$ is contained inside a different component of $\mathcal{F}(S_n)$ by Proposition \ref{approximation_step}. Thus $D(z,\varepsilon)$ must contain a point in $\mathcal{J}=\mathcal{J}(S_n)$. This proves that each point $z\in J$ is close to $\mathcal{J}$. Conversely, if $z\in\mathcal{J}$, then Proposition \ref{approximation_step} ensures that $z\not\in\cup_j A_j^\varepsilon$, and hence $z$ is close to $J$.


It remains to prove that $d_H(A_j,\mathcal{A}_j)<\varepsilon$. Let $z\in A_j$. By Lemma \ref{good_haus_approx}, $z$ is close to $A_j^\varepsilon$. Since $A_j^\varepsilon\subset \mathcal{A}_j$, we have that $z$ is close to $\mathcal{A}_j$. Conversely, if $z\in\mathcal{A}_j\setminus A_j$, then $z\in\mathcal{A}_j\setminus\cup_jA_j^\varepsilon$ (since each $A_j^\varepsilon$ is in the basin of attraction for $\zeta_j$). Thus there is a point inside $J=\partial A_j$ close to $z$, and hence a point inside $A_j$ close to $z$. This proves each $z\in \mathcal{A}_j$ is close to $A_j$ and hence $d_H(A_j,\mathcal{A}_j)$ is small. 
\end{proof}


\end{document}